\documentclass[12pt,a4paper]{amsart}
\usepackage{color,graphicx,fullpage,amssymb}
\usepackage{url}
\usepackage[colorlinks=true]{hyperref}
\newtheorem{theorem}{Theorem}[section]
\newtheorem{lemma}[theorem]{Lemma}
\newtheorem{proposition}[theorem]{Proposition}
\newtheorem{corollary}[theorem]{Corollary}

\newtheorem{maintheorem}{Theorem}

\theoremstyle{definition}
\newtheorem{definition}[theorem]{Definition}
\newtheorem{remark}[theorem]{Remark}
\newtheorem{example}[theorem]{Example}

\newcommand{\Q}{\mathbb{Q}}
\newcommand{\Qp}{\mathbb{Q}_p}

\newcommand{\Zp}{\mathbb{Z}_p}
\newcommand{\N}{\mathbb{N}}
\newcommand{\R}{\mathbb{R}}
\newcommand{\Z}{\mathbb{Z}}

\newcommand{\dd}{\mathrm{d}}

\newcommand{\e}{\mathrm{e}}
\newcommand{\g}{\mathfrak{g}}
\newcommand{\M}{\mathcal{M}}

\newcommand{\X}{\mathfrak{X}}
\newcommand{\sphere}{\mathrm{S}^2_p}
\newcommand{\Circle}{\mathrm{S}^1_p}
\newcommand{\ovcircle}{\overline{\mathrm{S}}^1_p}
\newcommand{\hatcircle}{\widehat{\mathrm{S}}^1_p}
\newcommand{\G}{\mathrm{G}_p}
\newcommand{\T}{\mathrm{T}}
\newcommand{\tr}{^{\mathrm{T}}}
\newcommand{\letnpos}{Let $n$ be a positive integer}
\newcommand{\letpprime}{Let $p$ be a prime number}
\renewcommand{\le}{\leqslant}
\renewcommand{\ge}{\geqslant}
\AtBeginEnvironment{pmatrix}{\let\frac\dfrac}
\DeclareMathOperator{\ord}{ord}

\DeclareMathOperator{\GL}{GL}
\DeclareMathOperator{\Ad}{Ad}
\DeclareMathOperator{\ad}{ad}
\DeclareMathOperator{\im}{im}
\numberwithin{equation}{section}

\newenvironment{enumerate-roman}{\begin{enumerate}
		
	}{\end{enumerate}}

\newenvironment{enumerate-alph}{\begin{enumerate}
		
	}{\end{enumerate}}

\title{Group actions on $p$-adic symplectic manifolds}
\author[Luis Crespo, \'Alvaro Pelayo]{Luis Crespo\,\,\,\,\,\, \'Alvaro Pelayo}

\address{Luis Crespo,
	Departamento de Matem\'{a}ticas, Estad\'{i}stica y Computaci\'{o}n, Universidad de Cantabria, Av.~de Los Castros 48, 39005 Santander, Spain}
\email{luis.cresporuiz@unican.es}

\address{\'Alvaro Pelayo,
	Facultad de Ciencias Matem\'aticas,
	Universidad Complutense de Madrid, 28040 Madrid, Spain, and Real Academia de Ciencias Exactas, F\'isicas y Naturales, Madrid, Spain}
\email{alvpel01@ucm.es}

\begin{document}
	
\begin{abstract}
	\letpprime. We introduce symplectic actions of $p$-adic analytic Lie groups on $p$-adic symplectic manifolds. Then we show that any $p$-adic symplectic action $G\times(M,\omega)\to(M,\omega)$ has a momentum map $\mu:M\to\g^*$, and that a proper $p$-adic symplectic action is Hamiltonian if and only if every orbit is isotropic. We conclude by defining $p$-adic symplectic toric manifolds, by analogy with the real case.
\end{abstract}
	
\maketitle

\section{Introduction}\label{sec:intro}

Symplectic geometry originates in the study of planetary motions. It is a fundamental subject in itself but also because of its many connections to other subjects such as classical and quantum physics, microlocal analysis, representation theory, integrable systems, etc., see \cite{HofZeh,OrtRat,Pelayo-symplectic,PelVuN-symplectic}. On the other hand the $p$-adic numbers play an important role in modern physics, see for example \cite{BFOW,FreOls,FreWit,FGZ,GarLop,Volovich} for applications to string dynamics/string theory and \cite{HSSS,MarTed} for applications in eternal inflation in cosmology.

On the other hand, group actions help us describing the symmetries of the physical world, and their study in a main area within real symplectic geometry. In the present paper we introduce $p$-adic symplectic and (weakly) Hamiltonian actions of $p$-adic analytic Lie groups. The definitions of weakly Hamiltonian and Hamiltonian actions on $p$-adic analytic manifolds are analogous to the ones in the real case. Nonetheless, because the field of $p$-adic numbers is so different from its real counterpart, the statements one can prove about such $p$-adic actions can be, as we will see, very different from the standard statements in the real case. In fact, our first two main results, which we state next (Theorems \ref{thm:weak}, \ref{thm:hamiltonian}) do not have a real analog (they are in fact, false, see Remark \ref{rem:counterexample-real}). Our first main result is:

\begin{maintheorem}[Every $p$-adic symplectic action has a momentum map]\label{thm:weak}
	\letpprime. Let $G$ be a $p$-adic analytic Lie group with Lie algebra $\g$. Let $(M,\omega)$ be a paracompact $p$-adic analytic symplectic manifold and let $\psi:G\times M\to M$ be a $p$-adic analytic symplectic Lie group action on $(M,\omega)$. Then $\psi:G\times M\to M$ is weakly Hamiltonian, that is, there exists a $p$-adic analytic momentum map $\mu:M\to\g^*$ for $\psi$.
\end{maintheorem}

Our second main result gives a characterization of $p$-adic Hamiltonian actions in terms of the properties of their orbits if $G$ is Abelian (for example the torus) and the action is proper (which includes the case when $G$ is compact):

\begin{maintheorem}[Characterization of $p$-adic Hamiltonian actions]\label{thm:hamiltonian}
	\letpprime. Let $G$ be an Abelian $p$-adic analytic Lie group with Lie algebra $\g$. Let $(M,\omega)$ be a paracompact $p$-adic analytic symplectic manifold and let $\psi:G\times M\to M$ be a proper $p$-adic analytic symplectic Lie group action on $(M,\omega)$. For every $\xi\in\g$, let $\mathrm{X}_\psi(\xi)$ be the infinitesimal generator of the $G$-action in the direction of $\xi$. Then $\psi:G\times M\to M$ is Hamiltonian if and only if $\omega(\mathrm{X}_\psi(\xi),\mathrm{X}_\psi(\eta))=0$ for every $\xi,\eta\in\g,$ that is, if and only if every $G$-orbit is isotropic.
\end{maintheorem}

Our proofs of Theorems \ref{thm:weak} and \ref{thm:hamiltonian} requires that $(M,\omega)$ is paracompact. We are not assuming this in general for $p$-adic analytic manifolds. Our main reference, Schneider's book on $p$-adic analytic Lie groups \cite{Schneider}, also does not assume paracompactness.

Theorem \ref{thm:hamiltonian} stands in strong contrast with the real case where such a result is false: there are many proper symplectic actions with Lagrangian orbits and which are not Hamiltonian, see \cite[Theorems 9.4 and 9.6]{DuiPel} and \cite[Theorem 8.2.1, case 4]{Pelayo}. See Figure \ref{fig:sets} for a comparison of the properties between the real and $p$-adic cases.

Theorem \ref{thm:hamiltonian} is false if we do not assume that $\psi$ is proper. It is difficult to find counterexamples because all Lie group actions which can be easily constructed are proper. In order to construct a non-proper one we would need two sequences of points in $M$, both convergent, such that one sequence can be converted into the other by applying a sequence of points in $G$ with no convergent subsequence (note that this is possible only for a non-compact $G$); at the same time we need to preserve the group operations and the symplectic structure.

\begin{maintheorem}[Example of non-Hamiltonian $p$-adic symplectic action]\label{thm:non-hamiltonian}
	There exists a free non-proper non-Hamiltonian $p$-adic analytic symplectic Lie group action of the field of $p$-adic numbers $\Qp$ on the Cartesian product of the ring of $p$-adic integers with itself $(\Zp)^2$.
\end{maintheorem}

\begin{figure}
	\includegraphics{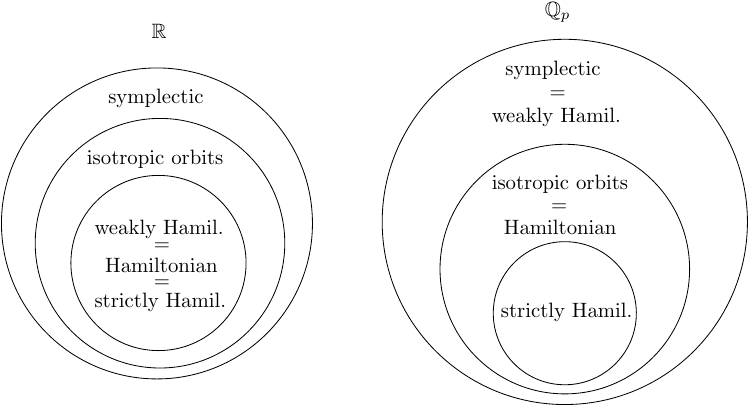}
	\caption{A comparison between the properties of proper Abelian Lie group actions on compact real symplectic manifolds (left) and proper Abelian Lie group actions on compact $p$-adic analytic symplectic manifolds (right), see Theorems \ref{thm:weak} and \ref{thm:hamiltonian} and Examples \ref{ex:oscillator}, \ref{ex:spin} and \ref{ex:translation2} for the diagram on he right-hand side.}
	\label{fig:sets}
\end{figure}

\medskip\noindent \textbf{Structure of the paper.} Section \ref{sec:prelim} introduces the necessary definitions about Lie groups and Lie group actions. Section \ref{sec:momentum} proves some results about $p$-adic analytic momentum maps and uses them to prove Theorem \ref{thm:weak}. Section \ref{sec:hamiltonian} proves Theorem \ref{thm:hamiltonian}. Section \ref{sec:torus} defines the $p$-adic torus and $p$-adic analytic symplectic toric manifolds. Section \ref{sec:examples} gives examples of $p$-adic analytic symplectic and Hamiltonian actions. Section \ref{sec:counterexample} gives an example of a $p$-adic analytic symplectic action which is not proper and not Hamiltonian despite having isotropic orbits, hence proving Theorem \ref{thm:non-hamiltonian}. Section \ref{sec:final} contains some remarks about our theorems. Finally, Appendix \ref{sec:appendix} recalls basic notions of $p$-adic geometry.

\medskip\noindent \textbf{Acknowledgments.} The second author was funded by a FBBVA (Bank Bilbao Vizcaya Argentaria Foundation) Grant for Scientific Research Projects with title \textit{From Integrability to Randomness in Symplectic and Quantum Geometry}. The second author thanks the Dean of the School of Mathematical Sciences Antonio Br\'u, and the Chair of the Department of Algebra, Geometry and Topology at the Complutense University of Madrid, Rutwig Campoamor, for their support and excellent resources he was provided with to carry out the FBBVA project.

\section{Symplectic, weakly Hamiltonian and Hamiltonian actions: definitions}\label{sec:prelim}

In this section we introduce the concepts of $p$-adic analytic Lie group and $p$-adic analytic Lie group action, as well as their properties, like being symplectic and Hamiltonian.

We refer to appendix \ref{sec:appendix} for a brief review of $p$-adic analytic manifolds and the basic concepts concerning group actions ($p$-adic or not).

\subsection{$p$-adic analytic Lie groups}

We start with the well-known notion of $p$-adic analytic Lie group.

\begin{definition}[$p$-adic analytic Lie group {\cite[section 13]{Schneider}}]
	\letpprime. Let $G$ be a $p$-adic analytic manifold endowed with a group structure. We say that $G$ is a \emph{$p$-adic analytic Lie group} if the product $G\times G\to G$ and inverse $G\to G$ functions are analytic in the topology of $G$ as a $p$-adic analytic manifold.
\end{definition}

Schneider's book \cite{Schneider} calls $p$-adic Lie groups what we call $p$-adic analytic Lie groups.

The \emph{Lie algebra} associated to $G$, denoted by $\g$, can be defined as the space of the right invariant vector fields, which can be understood as the space of the tangent vectors at the identity, identified, as a vector space, with $(\Qp)^k$, where $k$ is the dimension of $G$ as a $p$-adic analytic manifold.

Let $G$ be a $p$-adic analytic Lie group. Let $\g$ be the Lie algebra of $g$. For $g\in G$, we denote by $\Ad_g:\g\to\g$ the tangent map at the identity to the analytic map $G\to G$ which sends $h$ to $ghg^{-1}$. Let $\langle\cdot,\cdot\rangle$ denote the natural pairing between $\g^*$ and $\g$. For any $g\in G$ we define the map $\Ad^*_g:\g^*\to\g^*$ as follows. If $\eta\in\g^*$, we call $\Ad^*_g(\eta)$ the element of $\g^*$ defined by
\[\langle\xi,\Ad^*_g(\eta)\rangle=\langle\Ad_{g^{-1}}(\xi),\eta\rangle,\]
for every $\xi\in\g$. In particular, if $G$ is Abelian, $\Ad_g$ and $\Ad_g^*$ are the identity.

\subsection{$p$-adic analytic Lie group actions: symplectic and (weakly) Hamiltonian}

The following definitions introduce $p$-adic actions, as direct analogs of the notions in the real case.

\begin{definition}[$p$-adic symplectic Lie group action]
	\letpprime. Let $G$ be a $p$-adic analytic Lie group.
	
	\begin{itemize}
		\item Let $\mathbf{1}$ be the identity element of $G$. A \emph{$p$-adic analytic Lie group action} of $G$ on a $p$-adic analytic manifold $M$ is a $p$-adic analytic function $\psi:G\times M\to M$ such that $\psi(g_1,\psi(g_2,m))=\psi(g_1g_2,m)$ and $\psi(\mathbf{1},m)=m$, for every $g_1,g_2\in G$, $m\in M$.
		
		\item Let $(M,\omega)$ be a $p$-adic analytic symplectic manifold. A \emph{$p$-adic analytic symplectic Lie group action on} $(M,\omega)$ is a $p$-adic analytic Lie group action $\psi:G\times M\to M$ such that
		\[\psi(g,\cdot)^*\omega=\omega\]
		for every $g\in G$.
	\end{itemize}
\end{definition}

\begin{definition}[$p$-adic vector field generated by an action]	
	\letpprime. Let $G$ be a $p$-adic analytic Lie group with Lie algebra $\g$ and let $M$ be a $p$-adic analytic manifold. Given an action $\psi:G\times M\to M$ and $\xi\in\g$, we call $\mathrm{X}_\psi(\xi)$ the \emph{vector field generated by the action of $\psi$ in the direction of $\xi$} (or \emph{vector field of the infinitesimal action of $\xi$}). It can be defined by analogy with the real case as
	\[\mathrm{X}_\psi(\xi)(m)=\left.\frac{\dd}{\dd t}\psi(\exp(t\xi),m)\right|_{t=0}\]
	where the map $\exp:\g_\epsilon\to G$ is defined in \cite[Corollary 18.19]{Schneider}, and $\g_\epsilon$ is a small enough neighborhood of $0$ in $\g$. However, this definition (which is valid in both real and $p$-adic cases) can be rewritten in a way that, in the $p$-adic case, has the advantage of not needing the concept of exponential map of $p$-adic analytic Lie groups:
	\[\mathrm{X}_\psi(\xi)(m)=\psi(\cdot,m)_*(\xi).\]
\end{definition}

\begin{definition}[$p$-adic Hamiltonian Lie group action]\label{def:momentum-map}
	\letpprime. Let $G$ be a $p$-adic analytic Lie group with Lie algebra $\g$. Let $(M,\omega)$ be a $p$-adic analytic symplectic manifold. Let $\psi:G\times M\to M$ be a $p$-adic analytic symplectic Lie group action on $M$.
	
	\begin{itemize}
		\item We say that $\psi$ is \emph{weakly Hamiltonian} if there exists a $p$-adic analytic map \[\mu:M\to\g^*\] such that
		\begin{equation}\label{eq:hamilton}
			\omega(\mathrm{X}_\psi(\xi),\cdot)=\dd \mu_\xi,
		\end{equation}
		for each $\xi\in\g$ and $m\in M$, where $\mu_\xi:M\to\Qp$ is given by $\mu_\xi(m)=\langle \mu(m),\xi\rangle$.
		\item We say that $\mu:M\to\g^*$ is a \emph{$p$-adic analytic momentum map} for the $p$-adic analytic symplectic Lie group action of $G$ on $M$.
		\item We say that $\psi$ is \emph{Hamiltonian} if it is weakly Hamiltonian and \[\mu(\psi(g,m))=\Ad_g^*(\mu(m))\] for all $g\in G$ and $m\in M$. If $G$ is an Abelian group, this means that $\mu$ is constant on each orbit of the action, that is, $\mu(m)=\mu(m')$ if $m$ and $m'$ are related by the action of $G$. If the converse of the previous implication also holds, that is, $\mu(m)=\mu(m')$ if and only if $m$ and $m'$ are related by the action, we say that $\psi$ is \emph{strictly Hamiltonian}. See Figure \ref{fig:sphere} for a representation of the level curves of a $p$-adic analytic momentum map.
	\end{itemize}
\end{definition}

We often say ``Hamiltonian action'' instead of ``$p$-adic analytic symplectic Hamiltonian action'', ``momentum map'' instead of ``$p$-adic analytic momentum map'', etc. in order to shorten the terminology (mainly in the proofs).

\begin{remark}
	The notions in Definition \ref{def:momentum-map} are analogous to the real case. However, in the real case, the three conditions of being Hamiltonian, weakly Hamiltonian, and strictly Hamiltonian, are equivalent in some situations, for example if $M$ is compact and connected. In the $p$-adic case, by Theorem \ref{thm:weak}, every symplectic action on a paracompact manifold admits a momentum map, that is, it is weakly Hamiltonian; however, by Theorem \ref{thm:hamiltonian}, being Hamiltonian is more restrictive, and being strictly Hamiltonian is even more restrictive. See Figure \ref{fig:sets} for a comparison.
\end{remark}

\section{$p$-adic momentum maps and proof of Theorem \ref{thm:weak}}\label{sec:momentum}

In this section we study the problem of assigning a $p$-adic analytic momentum map to a given $p$-adic symplectic Lie group action on a $p$-adic analytic symplectic manifold. First we prove some results about $p$-adic analytic momentum maps in general, and we use them to prove Theorem \ref{thm:weak}.

\subsection{Properties of $p$-adic momentum maps}

As a first result, we can see that the momentum map is not unique. This also happens in the real case, where the momentum map is defined except for a translation; but in this case the possible transformations include many more than just translations.

\begin{proposition}\label{prop:translation}
	\letpprime. Let $(M,\omega)$ be a $p$-adic analytic symplectic manifold, let $G$ be an Abelian $p$-adic analytic Lie group with Lie algebra $\g$ and let $\psi:G\times M\to M$ be a weakly Hamiltonian $p$-adic analytic symplectic Lie group action.
	\begin{enumerate}
		\item Two $p$-adic analytic momentum maps $\mu,\mu':M\to\g^*$ for $\psi:G\times M\to M$ differ by a piecewise constant function.
		\item Two $p$-adic analytic momentum maps $\mu,\mu':M\to\g^*$ for $\psi:G\times M\to M$ which make the action Hamiltonian differ by the lifting to $M$ of a function $M/G\to\g^*$ which is piecewise constant.
		\item Two $p$-adic analytic momentum maps $\mu,\mu':M\to\g^*$ for $\psi:G\times M\to M$ which make the action strictly Hamiltonian differ by a bijective piecewise translation of the image.
	\end{enumerate}
\end{proposition}

\begin{proof}
	Part (1) is a direct consequence of the definition: the difference $\mu'-\mu$ must have zero differential at all points, hence it is piecewise constant; and adding a piecewise constant function to a momentum map gives another possible momentum map.
	
	Part (2) follows from part (1): if $\mu'-\mu$ is piecewise constant and $\mu$ and $\mu'$ are constant along each fiber, then $\mu'-\mu$ is also constant along each fiber, and it is the lifting to $M$ of a piecewise constant function on $M/G$.
	
	Part (3) follows from part (2): if $\mu$ and $\mu'$ make the action strictly Hamiltonian, then they descend to injective functions from $M/G$ to $\g^*$, which means that the composition $\mu'\circ\mu^{-1}$, as a function from the image of $\mu$ to the image of $\mu'$, is bijective. By part (2), it must be a piecewise translation.
\end{proof}

The following are the $p$-adic analogs of \cite[Lemmas 5.2.1 and 5.2.3]{McDSal}, where the brackets $\{\cdot,\cdot\}$ and $[\cdot,\cdot]$ are defined by analogy with the real case.

\begin{lemma}\label{lemma:poisson}
	\letpprime. Let $G$ be a $p$-adic analytic Lie group with Lie algebra $\g$ and let $(M,\omega)$ be a $p$-adic analytic symplectic manifold. Let $\psi:G\times M\to M$ be a $p$-adic analytic symplectic Hamiltonian action with $p$-adic analytic momentum map $\mu:M\to\g^*$. For every $\xi\in\g$, let $\mu_\xi:M\to\Qp$ be as defined in Definition \ref{def:momentum-map}. Then, for any $\xi,\eta\in\g$,
	\[\{\mu_\xi,\mu_\eta\}=\mu_{[\xi,\eta]}.\]
\end{lemma}

\begin{proof}
	Let $g:\Zp\to G$ such that $g(0)=\mathbf{1}$ and $\dot{g}(0)=\eta$. Let $m\in M$. Since $\psi$ is Hamiltonian,
	\[\mu_\xi(\psi(g(t),m))=\langle\mu(\psi(g(t),m)),\xi\rangle=\langle\Ad_{g(t)}^*(\mu(m)),\xi\rangle=\langle\mu(m),\Ad_{g(t)}(\xi)\rangle.\]
	By differentiating with respect to $t$, we get
	\[\dd\mu_\xi(\mathrm{X}_\psi(\eta)(m))=\langle\mu(m),[\xi,\eta]\rangle,\]
	that is
	\[\{\mu_\xi,\mu_\eta\}(m)=\mu_{[\xi,\eta]}(m)\]
	for all $m\in M$, as we wanted to prove.
\end{proof}

\begin{lemma}\label{lemma:cocycle}
	\letpprime. Let $G$ be a $p$-adic analytic Lie group with Lie algebra $\g$ and let $(M,\omega)$ be a $p$-adic analytic symplectic manifold. Let $\psi:G\times M\to M$ be a $p$-adic analytic symplectic weakly Hamiltonian action with $p$-adic analytic momentum map $\mu:M\to\g^*$. Then, there exists $\tau:\g\times\g\to\Omega^0(M)$ which satisfies the following conditions:
	\begin{itemize}
		\item $\tau$ is bilinear;
		\item for any $\xi,\eta\in\g$, $\tau(\xi,\eta)$ is locally constant;
		\item for any $\xi,\eta,\zeta\in\g$,
		\[\tau([\xi,\eta],\zeta)+\tau([\eta,\zeta],\xi)+\tau([\zeta,\xi],\eta)=0;\]
		\item for any $\xi,\eta\in\g$,
		\[\{\mu_\xi,\mu_\eta\}-\mu_{[\xi,\eta]}=\tau(\xi,\eta).\]
	\end{itemize}
\end{lemma}

\begin{proof}
	Let $\tau(\xi,\eta)=\{\mu_\xi,\mu_\eta\}-\mu_{[\xi,\eta]}$. Then we have that $\tau$ is bilinear, and
	\[\mathrm{X}_{\{\mu_\xi,\mu_\eta\}}=[\mathrm{X}_{\mu_\xi},\mathrm{X}_{\mu_\eta}]=[\mathrm{X}_\psi(\xi),\mathrm{X}_\psi(\eta)]=\mathrm{X}_\psi([\xi,\eta])=\mathrm{X}_{\mu_{[\xi,\eta]}},\]
	which implies $\mathrm{X}_{\tau(\xi,\eta)}=0$ and $\tau(\xi,\eta)$ is locally constant. Finally, for $\xi,\eta,\zeta\in\g$,
	\begin{align*}
		& \tau([\xi,\eta],\zeta)+\tau([\eta,\zeta],\xi)+\tau([\zeta,\xi],\eta) \\
		& =\{\mu_{[\xi,\eta]},\mu_\zeta\}-\mu_{[[\xi,\eta],\zeta]}+\{\mu_{[\eta,\zeta]},\mu_\xi\}-\mu_{[[\eta,\zeta],\xi]}+\{\mu_{[\zeta,\xi]},\mu_\eta\}-\mu_{[[\zeta,\xi],\eta]} \\
		& =\{\mu_{[\xi,\eta]},\mu_\zeta\}+\{\mu_{[\eta,\zeta]},\mu_\xi\}+\{\mu_{[\zeta,\xi]},\mu_\eta\} \\
		& =\{\{\mu_\xi,\mu_\eta\}-\tau(\xi,\eta),\mu_\zeta\}+\{\{\mu_\eta,\mu_\zeta\}-\tau(\eta,\zeta),\mu_\xi\}+\{\{\mu_\zeta,\mu_\xi\}-\tau(\zeta,\xi),\mu_\eta\} \\
		& =\{\{\mu_\xi,\mu_\eta\},\mu_\zeta\}+\{\{\mu_\eta,\mu_\zeta\},\mu_\xi\}+\{\{\mu_\zeta,\mu_\xi\},\mu_\eta\} \\
		& =0,
	\end{align*}
	where we are using Jacobi's identity in the third and sixth line, and the fact that the Poisson bracket of a locally constant function with any other function is zero in the fifth line.
\end{proof}

\begin{remark}\label{rem:tau}
	If we add a locally constant function $\sigma:M\to\g^*$ to $\mu$, the new value of $\tau$ is
	\[\tau'(\xi,\eta)(m)=\tau(\xi,\eta)(m)-\langle\sigma(m),[\xi,\eta]\rangle.\]
	If we want the action to be Hamiltonian, by Lemma \ref{lemma:poisson}, we need to make $\tau'=0$, which means that
	\[\tau(\xi,\eta)(m)=\langle\sigma(m),[\xi,\eta]\rangle\]
	for some $\sigma:M\to\g^*$ locally constant. If $G$ is Abelian, this just means $\tau(\xi,\eta)=0$ or $\{\mu_\xi,\mu_\eta\}=0$, which is possible only if the orbits are isotropic, that is, the condition of Theorem \ref{thm:hamiltonian}. However, this is not enough to prove Theorem \ref{thm:hamiltonian}, because the isotropy of the orbits only implies $\{\mu_\xi,\mu_\eta\}=0$ when we want something stronger, namely, that $\mu$ is constant in each orbit. Lemma \ref{lemma:poisson} tells us that $\mu$ being constant in each orbit implies $\{\mu_\xi,\mu_\eta\}=0$, but not the converse: in the real case this needs that the group $G$ is connected, and in the $p$-adic case the only connected Lie group is the trivial group.
\end{remark}

The following statement is the $p$-adic analog of \cite[Lemma 5.2.5]{McDSal}, which is essentially identical to the statement in the real case.

\begin{lemma}\label{lemma:duals-and-orthogonals}
	\letpprime. Let $G$ be a $p$-adic analytic Lie group with Lie algebra $\g$ and let $(M,\omega)$ be a $p$-adic analytic symplectic manifold. Let $\psi:G\times M\to M$ be a $p$-adic analytic symplectic Hamiltonian group action with a $p$-adic analytic momentum map $\mu:M\to\g^*$. Let $\mathrm{I}_m:\T_mM\to \T_mM^*$ be the isomorphism induced by $\omega$ (that is, $\mathrm{I}_m(v)=\omega_m(v,\cdot)$) and let $\mathrm{L}_m:\g\to \T_mM$ be given by $\mathrm{L}_m(\xi)=\mathrm{X}_\psi(\xi)(m)$. Then the following hold.
	\begin{enumerate}
		\item The dual map of $\T_m\mu:\T_mM\to\g^*$ as a linear map is given by
		\[\T_m\mu^*=\mathrm{I}_m\circ \mathrm{L}_m:\g\to\T_mM^*.\]
		\item The symplectic complement, that is, the complement with respect to the $p$-adic symplectic form, of the kernel of $\T_m\mu$ is the tangent space to the orbit of $m$:
		\[(\ker\T_m\mu)^\omega=\im \mathrm{L}_m.\]
		\item Let $\xi\in\g$. Let $\ad(\xi):\g\to\g$ be given by $\ad(\xi)=[\xi,\cdot]$. Denote by $\ad(\xi)^*:\g^*\to\g^*$ the dual linear map of $\ad(\xi)$. Then
		\[\T_m\mu(\mathrm{L}_m(\xi))=-\ad(\xi)^*(\mu(m)).\]
	\end{enumerate}
\end{lemma}

\begin{proof}
	\begin{enumerate}
		\item This follows from the fact that, for any $v\in\T_m M$,
		\[\langle\T_m\mu^*(\xi),v\rangle=\langle \xi,\T_m\mu(v)\rangle=\dd\mu_\xi(m)(v)=\omega_m(\mathrm{X}_\psi(\xi)(m),v)=\langle \mathrm{I}_m(\mathrm{L}_m(\xi)),v\rangle.\]
		\item First we have that, for $u\in\ker\T_m\mu$ and $v\in\im \mathrm{L}_m$, there is $\xi$ such that $\mathrm{X}_\psi(\xi)(m)=v$, and
		\[\omega_m(u,v)=\omega_m(u,\mathrm{X}_\psi(\xi)(m))=-\dd\mu_\xi(m)(u)=0.\]
		This implies that $(\ker\T_m\mu)^\omega=\im \mathrm{L}_m$. Moreover, by part (1),
		\[\dim\im \mathrm{L}_m=\dim\im\T_m\mu=\dim M-\dim\ker\T_m\mu=\dim(\ker\T_m\mu)^\omega,\]
		hence the two spaces must coincide.
		\item For any $\xi,\eta\in\g$,
		\[\langle\T_m\mu(\mathrm{L}_m(\xi)),\eta\rangle=\dd\mu_\eta(m)(\mathrm{X}_\psi(\xi)(m))=\{\mu_\eta,\mu_\xi\}(m).\]
		By Lemma \ref{lemma:poisson},
		\[\langle\T_m\mu(\mathrm{L}_m(\xi)),\eta\rangle=\mu_{[\eta,\xi]}(m)=-\mu_{\ad(\xi)(\eta)}(m)=\langle-\ad(\xi)^*(\mu(m)),\eta\rangle,\]
		as we wanted.\qedhere
	\end{enumerate}
\end{proof}

The following lemma writes the condition of being a momentum map in terms of the coordinate functions of $\mu$.

\begin{lemma}\label{lemma:coordinates}
	\letpprime. Let $k$ be a positive integer. Let $G$ be a $p$-adic analytic Lie group of dimension $k$. Let $(M,\omega)$ be a $p$-adic analytic symplectic manifold, let $\psi:G\times M\to M$ be a $p$-adic analytic symplectic Lie group action, and let $\mu:M\to\g^*$ be any $p$-adic analytic function. Let $\{\xi_1,\ldots,\xi_k\}$ be a basis of $\g$. Then $\mu$ is a $p$-adic analytic momentum map for $\psi:G\times M\to M$ if and only if, for all $m\in M$ and $i\in\{1,\ldots,k\}$,
	\begin{equation}\label{eq:hamilton2}
		\dd\mu_{\xi_i}=\omega_m(\mathrm{X}_\psi(\xi_i)(m),\cdot)
	\end{equation}
\end{lemma}

\begin{proof}
	If $\mu$ is a momentum map for $\psi$, then Hamilton's equations \eqref{eq:hamilton} must hold in particular for $\xi=\xi_i$:
	\[\omega(\mathrm{X}_\psi(\xi_i),\cdot)=\dd \mu_{\xi_i}, i\in\{1,\ldots,k\},\]
	as we wanted.
	
	Conversely, if equation \eqref{eq:hamilton2} holds for all $m$ and $i\in\{1,\ldots,k\}$, then given $\xi\in\g^*$, we can write $\xi=\sum_{i=1}^k a_i\xi_i$ for $a_i\in\Qp$, and for $m\in M$ and $v\in\T_m M$, we have that
	\begin{align*}
		\omega_m(\mathrm{X}_\psi(\xi)(m),v) & =\omega_m\left(\mathrm{X}_\psi(\sum_{i=1}^k a_i\xi_i)(m),v\right) \\
		& =\omega_m\left(\psi(\cdot,m)_*(\sum_{i=1}^k a_i\xi_i),v\right) \\
		& =\sum_{i=1}^k a_i\omega_m(\psi(\cdot,m)_*(\xi_i),v) \\
		& =\sum_{i=1}^k a_i\omega_m(\mathrm{X}_\psi(\xi_i)(m),v) \\
		& =\sum_{i=1}^k a_i v(\mu_{\xi_i})=v(\sum_{i=1}^k a_i\mu_{\xi_i})=v(\mu_\xi)=\dd \mu_{\xi_i}(v),
	\end{align*}
	hence $\mu$ is a momentum map for $\psi$, as we wanted.
\end{proof}

For dimension $1$, Lemma \ref{lemma:coordinates} is a rewriting of the definition. In higher dimensions, for example if $G$ is the torus $(\Circle)^k$, Lemma \ref{lemma:coordinates} says that the coordinate functions of $G$ are the Hamiltonians for the actions of the generators of the torus.

The following lemma relates closed and exact forms on a $p$-adic ball.

\begin{lemma}\label{lemma:closed-to-exact}
	\letnpos. \letpprime. Let $\alpha$ be a closed $1$-form on $(\Zp)^n$ such that the components of $\alpha$ are given as power series each one converging in $(\Zp)^n$. Then \[\alpha=\dd f\] for some power series $f$ converging in $(\Zp)^n$.
\end{lemma}

\begin{proof}
	Let $\alpha=\sum_{i=1}^n \alpha_i\dd x_i$, where $(x_1,\ldots,x_n)$ are coordinates on $(\Zp)^n$ and $\alpha_i:(\Zp)^n\to\Qp$ is a power series. Since $\alpha$ is closed, we have that
	\[\frac{\partial\alpha_i}{\partial x_j}=\frac{\partial\alpha_j}{\partial x_i}\]
	for all $i,j\in\{1,\ldots,n\}$.
	
	We prove the lemma by induction on $n$. For $n=1$, we can take a power series $f$ converging on $\Zp$ and whose derivative gives $\alpha_1$; this power series always exists and it will converge in $(\Zp)^n$ if the original series converges in $(\Zp)^n$. Now $\dd f=\alpha_1$, as we wanted.
	
	Now we suppose the result holds for $n$ and we prove it for $n+1$. Let $S$ be the subset of points in $(\Zp)^{n+1}$ whose last coordinate is $0$, and let $\beta$ be the $1$-form obtained from $\alpha$ by restricting it to $S$ and removing the last component. By induction hypothesis, $\beta=\dd f$ for some $f:S\to\Qp$.
	
	We extend $f$ to $(\Zp)^{n+1}$ in such a way that
	\[\frac{\partial f}{\partial x_{n+1}}=\alpha_{n+1}.\]
	We now have to check that
	\[\frac{\partial f}{\partial x_i}=\alpha_i\]
	for all $i\in\{1,\ldots,n\}$. Let $m\in(\Zp)^n$ and $m'\in S$ obtained from $m$ by making the last coordinate zero. We have that
	\begin{align*}
		\frac{\partial}{\partial x_{n+1}}\left(\frac{\partial f}{\partial x_i}-\alpha_i\right) & =\frac{\partial^2 f}{\partial x_i\partial x_{n+1}}-\frac{\partial\alpha_i}{\partial x_{n+1}} \\
		& =\frac{\partial\alpha_{n+1}}{\partial x_i}-\frac{\partial\alpha_i}{\partial x_{n+1}}=0.
	\end{align*}
	Since $\partial f/\partial x_i-\alpha_i$ is a power series and its derivative respect to $x_{n+1}$ is zero, it must be constant along the line through $m$ and $m'$. Its value at $m'$ is zero by induction hypothesis, hence it is also zero at $m$ and we are done.
\end{proof}

Now we are ready to prove our first main result.

\subsection{Proof of Theorem \ref{thm:weak}}

Let $2n$ be the dimension of $M$. Let $\{\xi_1,\ldots,\xi_k\}$ be a basis for $\g$. By Lemma \ref{lemma:coordinates}, we want to find $\mu:M\to\g^*$ such that \eqref{eq:hamilton2} holds. By \cite[Corollary 3.2]{CrePel-JC}, we can write $M$ as a disjoint union of $p$-adic balls, and the problem reduces to finding $\mu$ for each ball. This, in turn, is equivalent to finding the $k$ coordinate functions $\mu_i=\mu_{\xi_i}$, for $i\in\{1,\ldots,k\}$.
	
Now, we want to find $\mu_i$ defined on a ball such that $\dd\mu_i=\omega(\mathrm{X}_\psi(\xi_i),\cdot)$. Since $\psi$ is symplectic, $\mathrm{X}_\psi(\xi_i)$ preserves $\omega$, and the right-hand side $\omega(\mathrm{X}_\psi(\xi_i),\cdot)$ is a closed $1$-form. We may assume that the components of this form are given by power series converging in the ball; otherwise, we can divide the ball into smaller balls where this property holds. By Lemma \ref{lemma:closed-to-exact}, there exists on each ball a function $\mu_i$ such that $\dd\mu_i=\omega(\mathrm{X}_\psi(\xi_i),\cdot)$, and the result is proved.

\section{Proof of Theorem \ref{thm:hamiltonian}: $p$-adic momentum maps and isotropic orbits}\label{sec:hamiltonian}

Suppose that $\psi$ is Hamiltonian. Let $\xi,\eta\in\g$ and define $\mu_\xi(m)=\langle\mu(m),\xi\rangle$ and $\mu_\eta(m)=\langle\mu(m),\eta\rangle$. Since $\psi$ is Hamiltonian, $\mu$ is preserved by the action of $G$, and $\dd \mu_\xi(\mathrm{X}_\psi(\eta))=0$. By Hamilton's equation, this implies $\omega(\mathrm{X}_\psi(\xi),\mathrm{X}_\psi(\eta))=0$, as we wanted.
	
Suppose now that $\omega(\mathrm{X}_\psi(\xi),\mathrm{X}_\psi(\eta))=0$ for all $\xi,\eta\in\g$. We must define a momentum map $\mu$ for $\psi$ which is constant on each orbit.
	
\smallskip\textit{Step 1: divide $M$ into balls such that $\psi$ is a power series in each ball.} Since $M$ is paracompact, $M$ is a disjoint union of $p$-adic balls, that is, there is a set $I$ such that
\[M=\bigcup_{i\in I}B_i\]
where the $B_i, i\in I,$ are pairwise disjoint balls. We can choose coordinates on each ball of the form $(x_1,y_1,\ldots,x_n,y_n)$ where $\dim M=2n$. (Actually, by \cite[Lemma 6.1]{CrePel-Darboux}, we can choose the balls so that they are symplectic, which means that \[\omega=\sum_{j=1}^n\dd x_j\wedge\dd y_j\] in the coordinates of the ball. This proof, however, does not need the balls to be symplectic.) We may assume, after dividing the balls if needed, that for each $i\in I$ there exists a neighborhood $H_i$ of the identity in $G$ such that, for $(g,m)\in H_i\times B_i$, the coordinates of $\psi(g,m)$ are given by power series in the coordinates of $g$ and $m$. In particular, $B_i$ is invariant by the action of $H_i$. See Figure \ref{fig:orbit} for an illustration of this step.
	
\begin{figure}
	\includegraphics{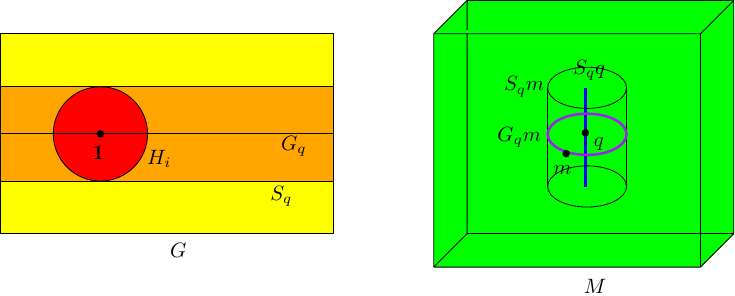}
	\caption{Step 1 of the proof of Theorem \ref{thm:hamiltonian}. The red circle around $\mathbf{1}\in G$ represents $H_i$. For $q\in M$, the horizontal line represents $G_q$, the stabilizer of $q$, and the orange band is $S_q=G_qH_i$. The blue line represents $S_qq$, which is the same as $H_iq$. We have marked a point $m$ near $q$; then the cylinder is $S_qm$ and the purple curve is $G_qm$ (when $m$ tends to $q$, the curve collapses to a point and the cylinder to the blue line).}
	\label{fig:orbit}
\end{figure}

\smallskip\textit{Step 2: reduce the size of the balls such that there are no bad pairs in the ball.} Let $q\in M$, and let $i\in I$ such that $q\in B_i$. We define, for each $k\in\N$,
\[U_{q,k}=\Big\{m\in B_i:\dd_{B_i}(q,m)\le p^{-k}\Big\},\]
where $\dd_{B_i}$ is the metric defined by the $p$-adic norm with respect to the coordinates $(x_1,y_1,\ldots,x_n,y_n)$ on $B_i$. Let $G_q$ be the stabilizer of $q$ and let \[S_q=G_qH_i.\] We call $(m,m')\in B_i\times B_i$ a \emph{bad pair} if $m,m'$ are related by the action of $G$ but not by the action of $S_q$, that is, there is $g\in G$ such that \[m'=\psi(g,m),\] but there is no $g\in S_q$ such that this happens.
	
Suppose that, for all $k\in\N$, there exists a bad pair $(m_k,m_k')$ contained in $U_{q,k}$ (see Figure \ref{fig:bad-pairs}). Then, for each $k\in\N$, there exists $g_k\in G$ such that \[m_k'=\psi(g_k,m_k).\] Both sequences $(m_k)_k$ and $(m_k')_k$ converge to $q$. Since the action is proper, there exists a convergent subsequence of $(g_k)_k$, which we still call $(g_k)_k$ (that is, we remove all terms not in the subsequence). Let $g_0$ be its limit and let \[g_k'=g_0^{-1}g_k.\] By continuity, $\psi(g_0,q)=q$ and $g_0\in G_q$. The sequence $(g_k')_k$ converges to the identity, hence $g_k'\in H_i$ for $k$ big enough. This implies \[g_k=g_0g_k'\in S_q\] for $k$ big enough, which contradicts the definition of bad pair. Therefore, for some $k$, $U_{q,k}$ does not contain bad pairs. We call \[U_q=U_{q,k}.\]

\begin{figure}
	\includegraphics{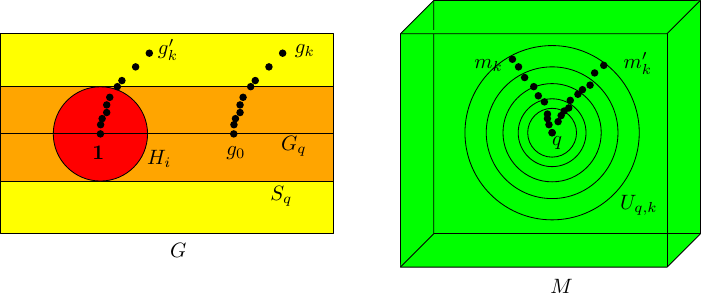}
	\caption{Step 2 of the proof of Theorem \ref{thm:hamiltonian}. For each $k\in\N$, $(m_k,m_k')$ is a bad pair contained in $U_{q,k}$, $g_k$ is such that $\psi(g_k,m_k)=m_k'$, $g_0$ is their limit, and $g_k'=g_0^{-1}g_k$.}
	\label{fig:bad-pairs}
\end{figure}
	
\smallskip\textit{Step 3: extract a covering by disjoint sets.} The sets $\{U_q\}_{q\in M}$ form an open covering of $M$. Since $M$ is paracompact, we can find a subcovering formed by disjoint open sets $\{U_q\}_{q\in Q}$.
	
\smallskip\textit{Step 4: define $\mu$ inductively.} Fix a good order $<$ in $Q$. Now we define $\mu$ successively on each $U_q$, following this order. Suppose we have defined $\mu$ on all sets $U_{q'}$ for $q'<q$, and we define it on $U_q$. We write \[U_q=U_q'\cup U_q'',\] where the points in $U_q'$ are related to those in $U_{q'}$ for some $q'<q$ by the action of $G$, and those in $U_q''$ are not. We will define $\mu$ first on $U_q'$ and then on $U_q''$.
	
\smallskip\textit{Step 4a: define $\mu$ on $U_q'$.} If $U_q'$ is empty, there is nothing to do. Otherwise, for each point $m\in U_q'$, there is $q'<q$ and $g\in G$ such that $\psi(g,m)\in U_{q'}$ (see Figure \ref{fig:Uq}). We define \[\mu(m)=\mu(\psi(g,m)).\] Since $\mu$ is constant in each orbit, the precise choice of $q'$ and $g$ does not affect the result. This extends $\mu$ to $U_q'$ in such a way that is still constant in each orbit.
	
\begin{figure}
	\includegraphics{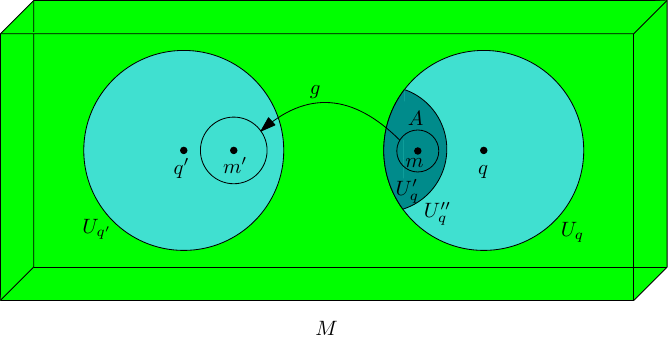}
	\caption{Step 4a of the proof of Theorem \ref{thm:hamiltonian}. The set $U_q$ is split into $U_q'$ and $U_q''$, and we are defining $\mu$ in $U_q'$ by copying its value from a previous set $U_{q'}$.}
	\label{fig:Uq}
\end{figure}

We must now prove that $\mu$ satisfies Hamilton's equations on $U_q'$. Let $m\in U_q'$. We choose $q'<q$ and $g\in G$ such that $\psi(g,m)\in U_{q'}$. Let \[m'=\psi(g,m).\] Then we have that, for $\xi\in\g$,
\begin{align*}
	\omega_m(\mathrm{X}_\psi(\xi)(m),\cdot) & =(\psi(g,\cdot)^*\omega_{m'})(\mathrm{X}_\psi(\xi)(m),\cdot) \\
	& =\psi(g,\cdot)^*\omega_{m'}(\psi(g,\cdot)_*(\mathrm{X}_\psi(\xi)(m)),\cdot) \\
	& =\psi(g,\cdot)^*\omega_{m'}(\mathrm{X}_\psi(\xi)(m'),\cdot) \\
	& =\psi(g,\cdot)^*\dd \mu_\xi(m'),
\end{align*}
where in the first line we are using that $\psi$ is symplectic, in the third that $G$ is Abelian, and in the fourth the induction hypothesis.
	
Let $A$ be a neighborhood of $m$ such that $A\subset U_q'$ and $\psi(g,A)\subset U_{q'}$. For each $m''\in A$, \[\mu(m'')=\mu(\psi(g,m'')).\] Hence, \[\mu_\xi(m'')=\mu_\xi(\psi(g,m''))\] and
\begin{align*}
	\omega_m(\mathrm{X}_\psi(\xi)(m),\cdot) & =\psi(g,\cdot)^*\dd \mu_\xi(m') \\
	& =\dd \mu_\xi(m),
\end{align*}
and we have extended the momentum map to $U_q'$. See Figure \ref{fig:Uq}.

\smallskip\textit{Step 4b: define $\mu$ in $U_q''$.} If $U_q''$ is empty, there is nothing to do. Otherwise, we start by defining a momentum map \[\mu':U_q\to\g^*\] for the action $\psi$ whose components are given by power series converging in $U_q$. We can do that with the same strategy of the proof of Theorem \ref{thm:weak}: we have that $U_q$ is contained in $B_i$ for some $i\in I$, and for $g$ near the identity and $m\in U_q\subset B_i$ the coordinates of $\psi(g,m)$ are power series in the coordinates of $g$ and $m$, so the same happens with $\omega(\mathrm{X}_\psi(\xi_j),\cdot)$ for each $j\in\{1,\ldots,k\}$.
	
If $U_q'$ is empty, we can just take $\mu=\mu'$ on $U_q$. Otherwise, let $\tau:U_q'\to\g^*$ be given by \[\tau=\mu'-\mu.\] Since both $\mu$ and $\mu'$ are momentum maps for $\psi$ when restricted to $U_q'$, by Proposition \ref{prop:translation}(1), $\tau$ is locally constant. We now extend $\tau$ to a locally constant function defined on all $U_q$ such that, if $m\in U_q''$ and $\psi(g,m)\in U_q''$ for some $g\in G$, then \[\tau(\psi(g,m))=\tau(m).\] Finally, we define \[\mu=\mu'-\tau\] on $U_q''$. By construction, $\mu$ satisfies Hamilton's equations on $U_q''$; in order to prove that the action of $G$ preserves $\mu$, since $U_q$ does not contain bad pairs, it is enough to show it for the action of $S_q$, that is, of $G_q$ and $H_i$. See Figure \ref{fig:Uq2} for this construction.

\begin{figure}
	\includegraphics{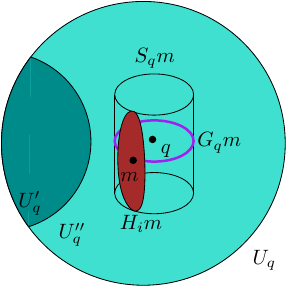}
	\caption{Step 4b of the proof of Theorem \ref{thm:hamiltonian}. The orbit of $m$ is contained in $U_q''$, and we need to prove that $\mu$ is constant along the purple curve $G_qm$ and in the brown set $H_im$.}
	\label{fig:Uq2}
\end{figure}
	
Let $m\in U_q''$. Since the components of $\mu'$ are given by power series in the coordinates on $U_q$, and the same happens with $\tau$ when restricted to the orbit of $m$ (because it is constant in the orbit), this also applies to $\mu$ in the orbit of $m$. For $g\in G_q$, since the action of $g$ preserves $\omega$ and $G$ is Abelian, this action also preserves the $1$-form $\omega(\mathrm{X}_\psi(\xi),\cdot)=\dd \mu_\xi$ for $\xi\in\g$, which means that the difference $\mu_\xi(\psi(g,m))-\mu_\xi(m)$ must be independent of $m$, and the same happens with $\mu(\psi(g,m))-\mu(m)$. But this difference is $0$ for $m=q$ because $\psi(g,q)=q$, so it must be constantly $0$ and \[\mu(\psi(g,m))=\mu(m).\] This proves the case of $g\in G_q$.
	
Now suppose $g\in H_i$. For $\xi,\eta\in\g$,
\begin{align*}
	\dd\langle\mu(\psi(\cdot,m)),\xi\rangle(g)(\eta) & =\dd [\mu_\xi(\psi(\cdot,m))](g)(\eta) \\
	& =\dd \mu_\xi(\psi(g,m))(\mathrm{X}_\psi(\eta)) \\
	& =\omega_{\psi(g,m)}(\mathrm{X}_\psi(\xi),\mathrm{X}_\psi(\eta)) \\
	& =0,
\end{align*}
and $\langle\mu(\psi(g,m)),\xi\rangle$ is given by a power series in the coordinates of $g$, hence it must be constant in $H_i$, and $\mu$ is constant in $H_im$, for each $m\in U_q''$, as we wanted to prove.

\section{The $p$-adic torus and $p$-adic symplectic toric manifolds}\label{sec:torus}

In this section we start by defining the $p$-adic torus and explaining its properties. We have already defined it in \cite{CrePel-JC} as the Cartesian product of the $p$-adic circle group $\Circle$ with itself. Then we define $p$-adic symplectic toric manifolds and give some examples. Finally we prove some partial results about the classification of $p$-adic symplectic toric manifolds.

\subsection{The $p$-adic torus}

In the real case, if a Lie group is compact, connected and Abelian, then it must be isomorphic to the torus $(\mathrm{S}^1)^k$, where $k$ is the dimension of the Lie group. In the $p$-adic case, a connected Lie group must be trivial; if we only assume compact and Abelian, there is in general not such isomorphism, and the torus is just one such group, which is not even compact if $p\equiv 1\mod 4$ (but in this case a compact subgroup can be taken instead). Another example of a compact Abelian Lie group is $(\Zp)^k$.

\begin{definition}
	Let $k$ be a positive integer. \letpprime. We define the \emph{$p$-adic circle} as
	\[\Circle=\Big\{(a,b)\in(\Qp)^2:a^2+b^2=1\Big\}.\]
	The \emph{$k$-dimensional $p$-adic torus} is defined as the $k$-th Cartesian product of the circle, that is, $(\Circle)^k$. We endow $\Circle$, and consequently $(\Circle)^k$, with an Abelian group structure where the product is given by
	\[(a,b)\cdot(a',b')=(aa'-bb',ab'+a'b),\]
	the identity element is $(1,0)$ and the inverse
	\[(a,b)^{-1}=(a,-b).\]
\end{definition}

\begin{proposition}\label{prop:lie}
	Let $k$ be a positive integer. \letpprime. The $k$-dimensional $p$-adic torus, endowed with the standard structure of the circle as a $p$-adic analytic manifold, is a $p$-adic analytic Lie group whose Lie algebra $\T_{\mathbf{1}}(\Circle)^k$ can be identified with $(\Qp)^k$ by the even-index coordinates.
\end{proposition}

\begin{proof}
	Being a Lie group follows from the fact that the product and the inverse in $(\Circle)^k$ are analytic. The Lie algebra can be defined as the vector space of the tangent vectors at the identity. In a neighborhood of $(1,0,1,0,\ldots,1,0)$, each point $(a_1,b_1,\ldots,a_k,b_k)$ in the torus can be identified with its coordinates $b_1,\ldots,b_k$, which means that the tangent space at that point is $(\Qp)^k$.
\end{proof}

\begin{proposition}[{\cite[Corollary 4.6]{CrePel-JC}}]\label{prop:structure-circle}
	\letpprime. The $p$-adic circle $\Circle$ is isomorphic to a product $\hatcircle\times\ovcircle\times p^d\Zp$, where $d$ is $2$ if $p=2$ and otherwise $1$, and $\hatcircle$ and $\ovcircle$ are discrete subgroups defined as follows:
	\begin{enumerate}
		\item $\hatcircle=\Z$ and $\ovcircle=\Z/(p-1)\Z$, if $p\equiv 1\mod 4$;
		\item $\hatcircle=\{\mathbf{1}\}$ and $\ovcircle=\Z/(p+1)\Z$, if $p\equiv 3\mod 4$;
		\item $\hatcircle=\{\mathbf{1}\}$ and $\ovcircle=\Z/4\Z$, if $p=2$.
	\end{enumerate}
	Moreover, the inclusion $p^d\Zp\hookrightarrow\Circle$ is given by $t\mapsto(\cos t,\sin t)$.
\end{proposition}

\begin{corollary}\label{cor:structure-torus}
	Let $k$ be a positive integer. \letpprime. The $k$-dimensional $p$-adic torus is isomorphic to
	\begin{enumerate}
		\item $(\Qp^*)^k$, if $p\equiv 1\mod 4$;
		\item $((\Z/(p+1)\Z)\times p\Zp)^k$, if $p\equiv 3\mod 4$;
		\item $((\Z/4\Z)\times 4\Z_2)^k$, if $p=2$.
	\end{enumerate}
\end{corollary}

If $p\equiv 1\mod 4$, the $p$-adic torus is isomorphic to $(\Qp^*)^k$, and hence non-compact. This is an inconvenience for defining toric manifolds. Indeed, consider an action of the $p$-adic torus on a compact manifold $M$. Let $m\in M$. The group $(\Qp^*)^k$ contains a discrete subgroup isomorphic to $\Z$, whose action on $m$ produces a sequence of points in $M$. Since $M$ is compact, this sequence contains a convergent subsequence, but $\Z$ contains no convergent subsequence, hence the action is not proper. That is to say, all actions of a $p$-adic torus on a compact manifold are non-proper if $p\equiv 1\mod 4$. To avoid this, we will restrict the action to the factors $\ovcircle\times p^d\Zp$.

\begin{definition}\label{def:G}
	Let $k$ be a positive integer. \letpprime. We denote by $\G$ the compact Lie group $\ovcircle\times p^d\Zp$, identified with a subgroup of $\Circle$. The \emph{compact $k$-dimensional $p$-adic torus} is given by $(\G)^k$.
\end{definition}

\begin{remark}
	In \cite{CrePel-nonsqueezing} we stated that the rotational action of the $p$-adic circle on the $p$-adic plane did not leave the ball $(\Zp)^2$ invariant and defined $\G$ as the largest subgroup of $\Circle$ which did leave it invariant. This is the same $\G$ as in Definition \ref{def:G}.
\end{remark}

\subsection{$p$-adic symplectic toric manifolds: definition}

The (real analog of the) following definition gives a fundamental class of Hamiltonian group actions, and we expect they will also play an important role in the $p$-adic category.

\begin{definition}[$p$-adic analytic symplectic toric manifold]
	Let $k$ be a positive integer. \letpprime. Let $(M,\omega)$ be a $p$-adic analytic symplectic manifold of dimension $2k$. Suppose that $(M,\omega)$ is endowed with a $p$-adic analytic effective symplectic Hamiltonian action of $(\G)^k$. The tuple $(M,\omega,\psi)$ is called a \emph{$2k$-dimensional $p$-adic analytic symplectic toric manifold}.
\end{definition}

For a Hamiltonian torus action in the real case, apart from translations, affine integral transformations in the image of $\mu$ can be achieved by a weak isomorphism in the torus action. The following is the $p$-adic equivalent. We refer to Appendix \ref{sec:appendix} for the notion of weakly isomorphic actions.

\begin{proposition}\label{prop:transformation}
	Let $k$ be a positive integer. \letpprime. Let $(M,\omega)$ be a $2k$-dimensional $p$-adic analytic symplectic manifold endowed with a $p$-adic symplectic Lie group action of the compact $k$-dimensional torus $\psi:(\G)^k\times M\to M$ with momentum map $\mu:M\to\g^*$, where $\g$ is the Lie algebra of $(\G)^k$. Let $A\in\GL(k,\Zp)$. Then there exists a $p$-adic symplectic action of the compact $k$-dimensional torus $\psi':(\G)^k\times M\to M$ weakly isomorphic to $\psi:(\G)^k\times M\to M$ whose momentum map is $A\mu:M\to\g^*$.
\end{proposition}

\begin{proof}
	We can write an element $g\in(\G)^k$ as \[g=g_1\cdot(\cos t_1,\sin t_1,\ldots,\cos t_k,\sin t_k),\] where $g_1\in\ovcircle$ and $t_1,\ldots,t_k\in p^d\Zp$. Then we define \[(t_1',\ldots,t_k')=(t_1,\ldots,t_k)A,\] \[\phi(g)=g_1\cdot(\cos t_1',\sin t_1',\ldots,\cos t_k',\sin t_k'),\] and finally $\psi'(g,m)=\psi(\phi(g),m)$.
	
	We can see that \[\mathrm{X}_{\psi'}(\xi)(m)=\psi'(\cdot,m)_*(\xi)=\psi(\cdot,m)_*(\phi_*(\xi))=\psi(\cdot,m)_*(A\tr\xi)=\mathrm{X}_\psi(A\tr\xi)(m).\] This leads to
	\[\omega(\mathrm{X}_{\psi'}(\xi),\cdot)=\omega(\mathrm{X}_\psi(A\tr\xi),\cdot)=\dd\langle\mu,A\tr\xi\rangle=\dd\langle A\mu,\xi\rangle,\]
	as we wanted.
\end{proof}

Hence, the momentum map of a $p$-adic analytic action can be considered up to piecewise constant functions and integral linear transformations (in the same way as, in the real case, the momentum maps are considered up to integral affine transformations).

Unlike in the real case \cite{Delzant}, there is not yet a classification of $p$-adic symplectic toric manifolds (there is not even an analog of the convex polytope of Atiyah \cite{Atiyah}, Kostant \cite{Kostant} and Guillemin-Sternberg \cite{GuiSte} in the real case).

\subsection{Results on $p$-adic symplectic toric actions}

Now we present some results about the classification of $p$-adic symplectic toric manifolds.

\begin{proposition}\label{prop:action}
	\letnpos. \letpprime. Let $(M,\omega)$ be a paracompact $2n$-dimensional $p$-adic symplectic manifold. There exists an effective and Hamiltonian $p$-adic analytic symplectic action of the compact $n$-dimensional $p$-adic torus on $(M,\omega)$.
\end{proposition}

\begin{proof}
	This is a consequence of a previous result \cite[Lemma 6.1]{CrePel-Darboux} which states that a paracompact $p$-adic analytic manifold is a disjoint union of symplectic balls. We take the rotation action of the compact $n$-dimensional $p$-adic torus on each $2n$-dimensional $p$-adic symplectic ball, as in Example \ref{ex:oscillator}; this gives us an effective and Hamiltonian action of the compact $n$-dimensional $p$-adic torus on the manifold.
\end{proof}

Regarding the classification of these manifolds, they have some invariants which are preserved by weak isomorphisms, that is, manifolds with different values of these invariants must belong to different classes, both modulo isomorphisms and modulo weak isomorphisms. Three of them are the volume of the manifold (which is preserved by any symplectomorphism, in particular by a weak isomorphism of symplectic toric manifolds), whether it is compact (also preserved by any symplectomorphism), and the number of fixed points (because the image of a fixed point by a weak isomorphism must be a fixed point). In the $p$-adic case, it turns out that there are many possible values for the number of fixed points, even after fixing the manifold.

\begin{proposition}\label{prop:many-actions}
	\letnpos. \letpprime. Let $(M,\omega)$ be a paracompact $2n$-dimensional $p$-adic symplectic manifold. There exist infinitely many $k\in\Z$ with $k\ge 0$ such that there exists an effective and Hamiltonian action of the compact $n$-dimensional $p$-adic torus on $(M,\omega)$ with exactly $k$ fixed points.
\end{proposition}

\begin{proof}
	We first observe that the strategy in the proof of Proposition \ref{prop:action} creates exactly one fixed point in each one of the symplectic balls in which the manifold is divided. Then, in order to vary the number of fixed points, we need to divide the manifold in a different number of symplectic balls. We can do this by dividing any of the balls into $p^{2n}$ symplectic balls with the next smaller radius, which increases the number of balls (and, consequently, that of fixed points) by $p^{2n}-1$. By repeating the construction, we can reach infinitely many different numbers of fixed points, as we wanted.
\end{proof}

Proposition \ref{prop:many-actions} implies that a classification of $p$-adic symplectic toric manifolds should include many different classes, because the manifolds in each class will have a common volume and a common number of fixed points, and we would have infinitely many classes for each volume, with different number of fixed points.

\section{Examples of $p$-adic symplectic/Hamiltonian actions}\label{sec:examples}

Here we give some examples of symplectic and Hamiltonian actions, such as the translation action, the rotation action in the plane or the sphere, or the angular momentum, and make some comments about being Hamiltonian and the fixed points.

\begin{example}\label{ex:translation}
	The simplest example of a Hamiltonian action is the action of $(\Zp)^n$ on $(\Zp)^{2n}$ given by translation on the odd coordinates:
	\[\psi((a_1,\ldots,a_n),(x_1,y_1,\ldots,x_n,y_n))=(x_1+a_1,y_1,\ldots,x_n+a_n,y_n).\]
	Its momentum map is $\mu(x_1,y_1,\ldots,x_n,y_n)=(y_1,\ldots,y_n)$, whose image is $(\Zp)^n$. The action is strictly Hamiltonian: two points in $(\Zp)^{2n}$ are related by the action if and only if they have the same $y$ coordinates. See Figure \ref{fig:maps} for a representation.
\end{example}

\begin{figure}
	\includegraphics[width=\linewidth]{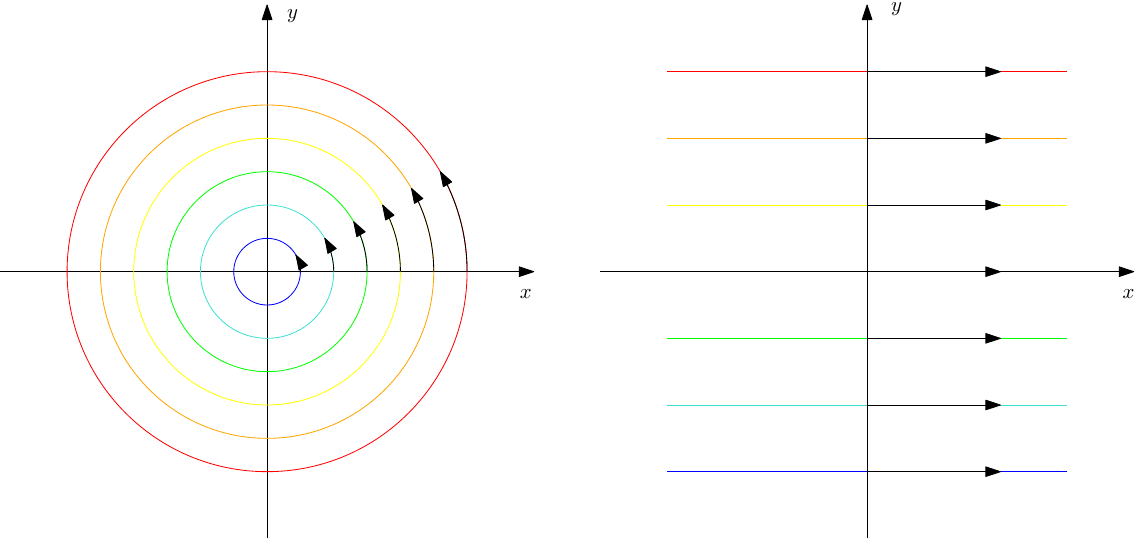}
	\caption{Level curves of the momentum maps of the rotation (left) and translation (right) actions in the plane. In the $p$-adic case like in the real case, the momentum map for the translation action is $y$ and for the rotation action it is $x^2+y^2$.}
	\label{fig:maps}
\end{figure}

\begin{example} \label{ex:oscillator}
	The rotational action of $\Circle$ on the $p$-adic plane $((\Qp)^2,\dd x\wedge\dd y)$ can be defined in the usual way:
	\[\psi((a,b),(x,y))=(ax+by,ay-bx).\]
	We have that $\mathrm{X}_\psi(1)=y\partial/\partial x-x\partial/\partial y$, which implies $\dd\mu(x,y)=x\dd x+y\dd y$ and \[\mu(x,y)=\frac{x^2+y^2}{2}\] (the action is sometimes defined with the opposite signs for $b$, in which case $\mu$ changes sign). The image of $\mu$ is all $\Qp$ if $p\equiv 1\mod 4$, the numbers with even $p$-adic order if $p\equiv 3\mod 4$, and the numbers with a digit $0$ at the left of the leading $1$ if $p=2$ (see \cite[Corollary 4.5]{CrePel-JC}). This is a Hamiltonian action, because $\mu$ is constant along each orbit. It is strictly Hamiltonian if and only if $p\not\equiv 1\mod 4$, because in that case two points with the same value of $\mu$ are related by the action of $\Circle$ (see \cite[Proposition 4.2]{CrePel-JC}).
	
	This action can be extended, by multiplication, to an action of the torus $(\Circle)^n$ on $(\Qp)^{2n}$. It does not restrict to an action on the ball $(\Zp)^{2n}$, but the action of the subgroup $(\G)^n$ restricts to the ball (in general, an action of a Lie subgroup of the same dimension has the same momentum map as the action of all the group). See Figure \ref{fig:maps} for a representation.
\end{example}

\begin{example}\label{ex:spin}
	The rotational action of $\Circle$ on the $p$-adic sphere $(\sphere,\omega)$, where
	\[\omega=-\frac{1}{x}\dd y\wedge\dd z=\frac{1}{y}\dd x\wedge\dd z=-\frac{1}{z}\dd x\wedge\dd y,\]
	can also be defined in the usual way (like the previous example, sometimes the signs for $b$ are the opposite):
	\[\psi((a,b),(x,y,z))=(ax+by,ay-bx,z).\]
	We have $\mathrm{X}_\psi(1)=y\partial/\partial x-x\partial/\partial y$, which implies $\dd\mu(x,y,z)=\dd z$ and \[\mu(x,y,z)=z.\] The image of $\mu$ is all $\Qp$ if $p\equiv 1\mod 4$; otherwise it is a rather complicated subset of $\Qp$, which in the case $p=2$ is contained in $\Zp$. Like the previous example, the action is always Hamiltonian, and it is strictly Hamiltonian if and only if $p\not\equiv 1\mod 4$. See \cite[section 5]{CrePel-JC} for more about this action.
\end{example}

\begin{figure}
	\includegraphics{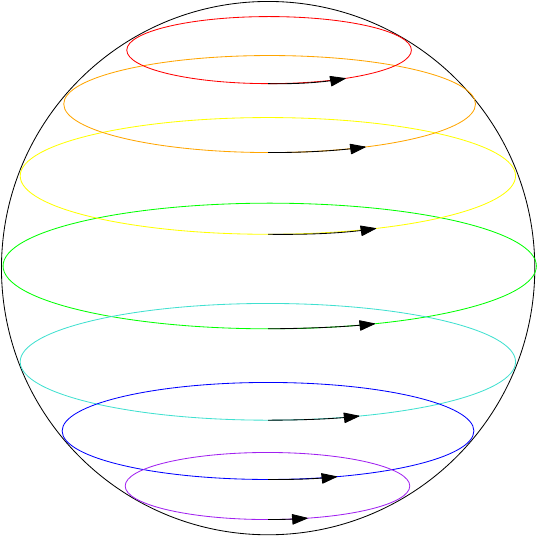}
	\caption{Level curves of the momentum map of the rotational action on the sphere, indicated by the arrows. The $p$-adic spin action is described in Example \ref{ex:spin}. Its momentum map, like in the real case, is $z$. See Figure \ref{fig:maps} for other examples of momentum maps.}
	\label{fig:sphere}
\end{figure}

\begin{example}\label{ex:composite}
	We can combine the actions in Examples \ref{ex:oscillator} and \ref{ex:spin} in order to obtain an action of $\Circle$ on $\sphere\times(\Qp)^2$ given by
	\[\psi((a,b),(x,y,z,u,v))=(ax-by,ay+bx,z,au-bv,av+bu).\]
	This leads to the momentum map
	\[\mu(x,y,z,u,v)=\frac{u^2+v^2}{2}+z.\]
	This function Poisson-commutes with $ux+vy$, resulting in the $p$-adic Jaynes-Cummings system which we studied in \cite{CrePel-JC}. The momentum map is surjective if $p\ne 2$. Another possibility is to combine the actions of $\Circle$ on two spheres: if we take $R_1\omega_1+R_2\omega_2$ as a symplectic form, where $\omega_1$ and $\omega_2$ are the forms in the two spheres, the momentum map of the action on the product is $R_1z_1+R_2z_2$, which is one of the components of the $p$-adic coupled angular momentum system \cite{CrePel-angmom}.
\end{example}

\begin{example}
	We now give an example of a Hamiltonian action of a group which is not Abelian. This is the $p$-adic analog of \cite[Example 5.3.1]{McDSal}. Consider the group $\mathrm{SO}(3,\Qp)$ of matrices $\Phi\in\M_3(\Qp)$ such that $\Phi\tr\Phi=I$ and $\det\Phi=1$. Its Lie algebra $\mathfrak{so}(3,\Qp)$ is given by the matrices $A\in\M_3(\Qp)$ such that $A+A\tr=0$. This vector space can be identified with $(\Qp)^3$ by the correspondence
	\[\xi=(\xi_1,\xi_2,\xi_3)\mapsto A_\xi=\begin{pmatrix}
		0 & -\xi_3 & \xi_2 \\
		\xi_3 & 0 & -\xi_1 \\
		-\xi_2 & \xi_1 & 0
	\end{pmatrix}.\]
	
	With this correspondence, the Lie bracket of $\mathfrak{so}(3,\Qp)$ becomes the cross product in $(\Qp)^3$: $A_{\xi\times\eta}=[A_\xi,A_\eta]$. We have also that the trace of $A_\xi\tr A_\eta$ is $2\langle\xi,\eta\rangle$, which means that the following diagram is commutative:
	\[\begin{array}{ccc}
		\mathfrak{so}(3,\Qp) & \xrightarrow{A\mapsto \frac{1}{2}\mathrm{trace}(A\tr\cdot)} & \mathfrak{so}(3,\Qp)^* \\
		\uparrow & & \downarrow \\
		(\Qp)^3 & \xrightarrow{\xi\mapsto \langle\xi,\cdot\rangle} & ((\Qp)^3)^*
	\end{array}\]
	where the up and down arrows are the correspondence $\xi\mapsto A_\xi$ and its dual. This allows us to identify $\mathfrak{so}(3,\Qp)^*$ also with $(\Qp)^3$. The adjoint action of $\Phi$ on $A_\xi$ gives $\Phi A_\xi\Phi^{-1}=A_{\Phi\xi}$, hence it becomes just a multiplicative action $\xi\mapsto\Phi\xi$ on $(\Qp)^3$, and the coadjoint action on $((\Qp)^3)^*$ is the dual of the action of $\Phi^{-1}$, that is, $\xi\mapsto(\Phi^{-1})\tr\xi=\Phi\xi$.
	
	We now define an action of $\mathrm{SO}(3,\Qp)$ on $(\Qp)^3\times(\Qp)^3$ by
	\[\psi(\Phi,(x,y))=(\Phi x,\Phi y).\]
	We take in $(\Qp)^3\times(\Qp)^3$ the symplectic form $\omega=\dd x_1\wedge\dd y_1+\dd x_2\wedge\dd y_2+\dd x_3\wedge\dd y_3$. Since $\Phi$ is orthogonal, its action preserves $\omega$:
	\[\psi(\Phi,\cdot)^*\omega=(\Phi\dd x)\tr\wedge\Phi\dd y=\dd x\tr\wedge\dd y=\omega.\]
	
	Let $A\in\mathfrak{so}(3,\Qp)$ and $\Phi:\Zp\to\mathrm{SO}(3,\Qp)$ such that $\Phi(0)=I$ and $\dot{\Phi}(0)=A$. Then
	\[\mathrm{X}_\psi(A)(x,y)=\left.\frac{\dd}{\dd t}(\Phi(t) x,\Phi(t) y)\right|_{t=0}=(Ax,Ay)\]
	and
	\[\omega(\mathrm{X}_\psi(A),\cdot)=\dd y\tr Ax-\dd x\tr Ay=\dd y\tr Ax+y\tr A\dd x=\dd\langle y,Ax\rangle.\]
	If $A=A_\xi$ for some $\xi\in(\Qp)^3$, we have $Ax=\xi\times x$ and
	\[\omega(\mathrm{X}_\psi(A_\xi),\cdot)=\dd\langle y,\xi\times x\rangle=\dd\langle x\times y,\xi\rangle.\]
	Hence we can take \[\mu(x,y)=x\times y.\] This action is Hamiltonian because
	\[\mu(\psi(\Phi,(x,y)))=\Phi x\times\Phi y=\Phi(x\times y)=\Ad_\Phi^*\mu(x,y),\]
	where the middle inequality holds because $\Phi$ is orthogonal of determinant $1$.
\end{example}

Unlike what happens in the real case, where the notions of weakly Hamiltonian action and Hamiltonian action are very similar, and in fact equivalent if $M$ is compact, in the $p$-adic case there are many weakly Hamiltonian actions which are not Hamiltonian, even if $M$ is compact.

\begin{example}\label{ex:translation2}
	Consider the action of $\Zp^2$ on $(\Zp^{2n},\sum_{i=1}^n\dd x_i\wedge\dd y_i)$, for all $n\ge 1$, by translation on the first two coordinates. The momentum map is given by \[\mu(x_1,y_1,\ldots,x_n,y_n)=(y_1,-x_1),\] which is not preserved by a translation. Actually the condition in Theorem \ref{thm:hamiltonian} is not satisfied:
	\[\omega(\mathrm{X}_\psi(\e_1),\mathrm{X}_\psi(\e_2))=\omega\left(\frac{\partial}{\partial x_1},\frac{\partial}{\partial y_1}\right)=1.\]
\end{example}

In contrast with the real case \cite{Frankel,McDuff}, there does not appear to be a relation between being Hamiltonian and having fixed points in the compact case, as the following two examples show.

\begin{example}\label{ex:translation1}
	The action in Example \ref{ex:translation} is Hamiltonian but it does not have fixed points. We can even achieve this for a group of dimension $1$: the action of $\Zp$ on the compact manifold $(\Zp)^{2n}$ by translation on the first coordinate $x_1$ is Hamiltonian, with momentum map $y_1$, but it does not have fixed points.
\end{example}

\begin{example}\label{ex:translation-partial}
	It is also possible to construct non-Hamiltonian symplectic actions on compact manifolds with fixed points: the action of $(p\Zp)^2$ on $(\Zp)^{2n}$ given by
	\[\psi((a,b),(x_1,y_1,\ldots,x_n,y_n))=\begin{cases}
		(x_1+a,y_1+b,x_2,y_2,\ldots,x_n,y_n) & \text{if } x_1,y_1\in p\Zp \\
		(x_1,y_1,\ldots,x_n,y_n) & \text{otherwise}
	\end{cases}\]
	has many fixed points and its momentum map is
	\[\mu(x_1,y_1,\ldots,x_n,y_n)=\begin{cases}
		(y_1,-x_1) & \text{if } x_1,y_1\in p\Zp \\
		0 & \text{otherwise,}
	\end{cases}\]
	which is not constant along any orbit different from a fixed point.
\end{example}

\section{A symplectic non-Hamiltonian $p$-adic action with isotropic orbits and proof of Theorem \ref{thm:non-hamiltonian}}\label{sec:counterexample}

The condition of being proper in Theorem \ref{thm:hamiltonian} cannot be removed. \letpprime. We now define a non-proper action of $\Qp$ on $(\Zp)^2$ which is not Hamiltonian, while the condition in Theorem \ref{thm:hamiltonian} always holds if $G$ has dimension $1$.

For $n\in\N$, let \[A_n=\Big\{0,\ldots,p^n-1\Big\}.\] We define a function $r_n:A_n\to A_n$ which reverses the digits of a number in base $p$:
\[a=\sum_{i=0}^{n-1} a_ip^i\Longrightarrow r(a)=\sum_{i=0}^{n-1}a_{n-1-i}p^i.\]
Let $S$ be the quotient of additive groups $\Qp/\Zp$. This quotient can be described as
\[S=\left\{\frac{a}{p^n}:n\in\N,a\in A_n\right\}\]
with the addition modulo $1$. An element $x\in\Qp$ can be written uniquely as $x=\lfloor x\rfloor+\{x\}$ where $\lfloor x\rfloor\in\Zp$ and $\{x\}\in S$. We define, for $n\in\N$ and $x\in \Zp$,
\[f_n(x)=p^n\{p^{-n}x\}\in A_n,\]
that is, $f_n(x)$ is the integer whose digits in base $p$ are the rightmost $n$ digits of $x$.

Now we define our action $\psi:\Qp\times(\Zp)^2\to(\Zp)^2$ as follows.
\begin{definition}\label{def:psi}
	Given $g\in\Qp$ and $(x,y)\in(\Zp)^2$, we first write $\{g\}$ as $a/p^n$ for some $n\in\N$ and $a\in A_n$. Let $b=f_n(y)$ and $c=a+r_n(b)$. If $c<p^n$ (that is, $c\in A_n$), we define
	\begin{equation}\label{eq:psi1}
		\psi(g,(x,y))=(x+\lfloor g\rfloor,y-b+r_n(c)).
	\end{equation}
	Otherwise, we must have $p^n\le c<2p^n$, and $c-p^n\in A_n$. We define
	\begin{equation}\label{eq:psi2}
		\psi(g,(x,y))=(x+\lfloor g\rfloor+1,y-b+r_n(c-p^n)).
	\end{equation}
	See Figure \ref{fig:counterexample} for a representation of this action.
\end{definition}

In order to prove that this $\psi$ is actually a symplectic Lie group action and that it is not Hamiltonian, we need some lemmas.

\begin{lemma}\label{lemma:psi-well-defined}
	The map $\psi$ in Definition \ref{def:psi} is well defined, that is, the result does not depend on the value of $n$ used.
\end{lemma}

\begin{proof}
Suppose we use $n$ and $n'$ with $n<n'$. When we write $\{g\}$ as $a/p^n$ and as $a'/p^{n'}$, we have $a'=p^{n'-n}a$. The number $b=f_n(y)$ consists, by definition, of the last $n$ digits of $b'=f_{n'}(y)$: we write $b'=p^nd+b$ for $d\in A_{n'-n}$. Then
\begin{align*}
	c' & =a'+r_{n'}(b') \\
	& =p^{n'-n}a+r_{n'}(p^nd+b) \\
	& =p^{n'-n}a+p^{n'-n}r_n(b)+r_{n'-n}(d) \\
	& =p^{n'-n}c+r_{n'-n}(d).
\end{align*}
If $c\in A_n$, then also $c'\in A_{n'}$ and
\begin{align*}
	y-b'+r_{n'}(c') & =y-p^nd-b+r_{n'}(p^{n'-n}c+r_{n'-n}(d)) \\
	& =y-p^nd-b+p^nd+r_n(c) \\
	& =y-b+r_n(c).
\end{align*}
If $c\notin A_n$, then also $c'\notin A_{n'}$ and
\begin{align*}
	y-b'+r_{n'}(c'-p^{n'}) & =y-p^nd-b+r_{n'}(p^{n'-n}(c-p^n)+r_{n'-n}(d)) \\
	& =y-p^nd-b+p^nd+r_n(c-p^n) \\
	& =y-b+r_n(c-p^n).
\end{align*}
Hence the result is actually the same.
\end{proof}

\begin{definition}
	For each $n\in\N$, we define $\phi_n:(\Zp)^2\to\Zp$ as
	\[\phi_n(x,y)=p^nx+r_n(f_n(y)),\]
	that is, the result of appending to $x$ the rightmost $n$ digits of $y$ reversed.
\end{definition}

\begin{figure}
	\includegraphics{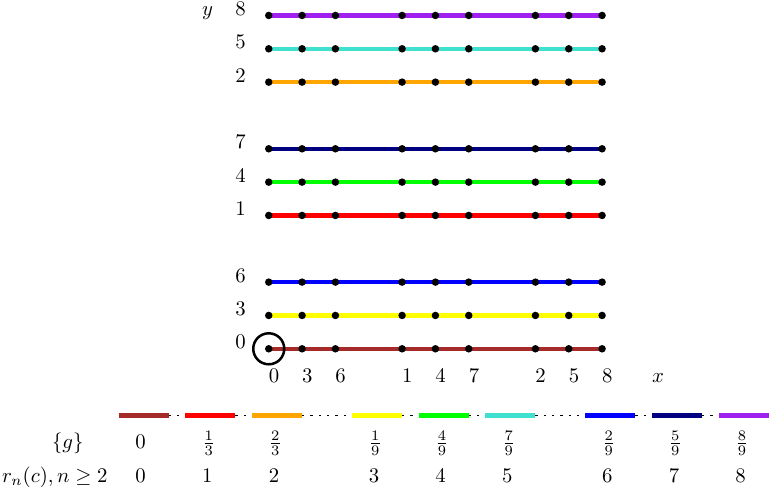}
	\caption{A representation of how $\psi(g,(0,0))$ varies with $g$. The points in the top part of the figure represent $(\Z_3)^2$, where each point is a ball of radius $1/9$. The circle indicates the initial point $(0,0)$. The nine segments in the bottom are nine balls of radius $1$ in $\Q_3$. When $g$ moves inside one of the segments, it changes the integer part while preserving the fractional part, and the effect on $\psi(g,(0,0))$ is moving it horizontally. The line along which it moves is determined by the value of $r_n(c)$, and in this case $c$ is the same as $a$ and depends on $\{g\}$; this line is indicated with the same color as the corresponding segment in the bottom part.}
	\label{fig:counterexample}
\end{figure}

\begin{lemma}\label{lemma:phi}
	For $x,y\in\Zp$ and $g\in p^{-n}\Zp$, $\phi_n(\psi(g,(x,y)))=p^ng+\phi_n(x,y)$.
\end{lemma}

\begin{proof}
	Since $g\in p^{-n}\Zp$, we can use this $n$ to compute $\psi(g,(x,y))$, and we have that
	\begin{align*}
		\phi_n(\psi(g,(x,y))) & =\phi_n(x+\lfloor g\rfloor+d,y-b+r_n(c-dp^n)) \\
		& =p^n(x+\lfloor g\rfloor+d)+r_n(f_n(y-b+r_n(c-dp^n))),
	\end{align*}
	where $d$ is $0$ or $1$ depending on which of the two equations \eqref{eq:psi1} or \eqref{eq:psi2} holds. By definition, $b$ consists on the last $n$ digits of $y$, hence $y-b$ ends in $n$ zeros, and
	\[f_n(y-b+r_n(c-dp^n))=r_n(c-dp^n),\]
	which implies
	\begin{align*}
		\phi_n(\psi(g,(x,y))) & =p^n(x+\lfloor g\rfloor+d)+r_n(r_n(c-dp^n)) \\
		& =p^n(x+\lfloor g\rfloor+d)+c-dp^n \\
		& =p^n(x+\lfloor g\rfloor)+c \\
		& =p^n(x+\lfloor g\rfloor)+a+r_n(b) \\
		& =p^n(x+\lfloor g\rfloor)+p^n\{g\}+r_n(f_n(y)) \\
		& =p^ng+\phi_n(x,y).\qedhere
	\end{align*}
\end{proof}

Now we prove that this action $\psi$ is what we need for Theorem \ref{thm:non-hamiltonian}. Actually, we will prove some more properties about this action.

\begin{theorem}\label{thm:psi-non-hamiltonian}
	The map $\psi:\Qp\times(\Zp)^2\to(\Zp)^2$ given in \eqref{eq:psi1} and \eqref{eq:psi2} satisfies the following conditions:
	\begin{itemize}
		\item $\psi$ is a symplectic Lie group action.
		\item $\psi$ is free, that is, the isotropy group of each point in $(\Zp)^2$ is trivial. In particular, $\psi$ does not have fixed points.
		\item $\psi$ is not proper.
		\item Two points $(x_1,y_1)$ and $(x_2,y_2)$ are in the same orbit of $\psi$ if and only if there exists $n\in\Z$ such that the digits of $y_1$ are the same than those of $y_2$ at the left of the position $n$.
		\item $\psi$ is not Hamiltonian.
	\end{itemize}
\end{theorem}

\begin{proof}
First we prove that $\psi$ is a symplectic Lie group action. We see that $\psi$ is smooth: $\{g\}$ as a function of $g$ and $f_n(y)$ as a function of $y$ are locally constant, hence the only local variation in $\psi(g,(x,y))$ occurs in the terms $x$, $\lfloor g\rfloor$ and $y$ of \eqref{eq:psi1} and \eqref{eq:psi2}. This means that $\psi(g,(x,y))$ is smooth and its differential in the variables $x$ and $y$ is the identity, which means that it preserves the symplectic form.

If $g=0$, then $a=0$, $c=r_n(b)$ and $r_n(c)=b$, which implies $\psi(0,(x,y))=(x,y)$. It is left to prove that $\psi(g_1+g_2,(x,y))=\psi(g_1,\psi(g_2,(x,y)))$.

Let $(x,y)\in(\Zp)^2$ and $g_1,g_2\in\Qp$. Let $n\in\N$ such that $g_1,g_2\in p^{-n}\Zp$. By Lemma \ref{lemma:phi}, we have that
\begin{align*}
	\phi_n(\psi(g_1+g_2,(x,y))) & =p^n(g_1+g_2)+\phi_n(x,y) \\
	& =\phi_n(\psi(g_1,\psi(g_2,(x,y)))).
\end{align*}
By the definition of $\phi_n$, this means that $\psi(g_1+g_2,(x,y))$ and $\psi(g_1,\psi(g_2,(x,y)))$ have the same first coordinate and the same rightmost $n$ digits of the second. Moreover, since $g_1,g_2\in p^{-n}\Zp$, applying $\psi$ does not change any digits of the second coordinate at the left of the position $n$. Hence, the two pairs are equal, and we are done.

Now we prove that $\psi$ is free. Let $g\in\Qp$ and $(x,y)\in(\Zp)^2$ such that $\psi(g,(x,y))=(x,y)$. Let $n\in\N$ such that $g\in p^{-n}\Zp$. Lemma \ref{lemma:phi} implies that
\[\phi_n(x,y)=\phi_n(\psi(g,(x,y)))=p^ng+\phi_n(x,y),\]
which implies $g=0$, as we wanted.

$\psi$ is not proper because, for all $n\in\N$,
\[\psi(p^{-n},(0,0))=(0,r_n(1))=(0,p^{n-1}),\]
and the limit of $p^{n-1}$ when $n$ tends to infinity is $0$, while the sequence $p^{-n}$ does not converge.

If two points $(x_1,y_1)$ and $(x_2,y_2)$ are in the same orbit, they are related by the action of $g\in\Qp$. If $g\in p^{-n}\Zp$, only the $n$ rightmost digits can change from $y_1$ to $y_2$. Conversely, if $y_1$ and $y_2$ only differ by the $n$ rightmost digits for some $n\in\N$, we take
\[g=p^{-n}(\phi_n(x_2,y_2)-\phi_n(x_1,y_1)).\]
Then, by Lemma \ref{lemma:phi},
\[\phi_n(\psi(g,(x_1,y_1)))=p^ng+\phi_n(x_1,y_1)=\phi_n(x_2,y_2).\]
This implies that $\psi(g,(x_1,y_1))$ and $(x_2,y_2)$ coincide in the first coordinate and $n$ digits of the second. The rest of digits of the second coordinate of $\psi(g,(x_1,y_1))$ are the same as those of $y_1$, which are those of $y_2$ by hypothesis. Hence $\psi(g,(x_1,y_1))=(x_2,y_2)$, as we wanted.

Since
\[\frac{\dd\lfloor g\rfloor}{\dd g}=1\text{ and }\frac{\dd\{g\}}{\dd g}=0,\]
the action $\psi$ is locally like a translation in the direction $x$, and a momentum map of $\psi$ is $(x,y)\mapsto y$. This is not constant on each orbit. Suppose there is another momentum map $\mu$ which is constant on the orbits. Without loss of generality, we may assume that $\mu(0,0)=0$. For all $n\in\N$, since $\mu$ is constant on the orbits and $\psi(p^{-n},(0,0))=(0,p^{n-1})$, $\mu(0,p^{n-1})=0$. But, by Proposition \ref{prop:translation}(1), there exists a neighborhood $U$ of the origin in $(\Zp)^2$ in which the two momentum maps coincide, that is, $\mu(x,y)=y$ in $U$; in particular, $\mu(0,p^n)=p^n$ for $n$ big enough, which is a contradiction. Hence this action is not Hamiltonian.
\end{proof}

\begin{remark}
	Despite being non-Hamiltonian, the action in Theorem \ref{thm:psi-non-hamiltonian} satisfies that $\{\mu_\xi,\mu_\eta\}=0$, because this equality just says that the orbits are isotropic, which always happens when they have dimension $1$. Hence, this example also serves as a counterexample to the converse of Lemma \ref{lemma:poisson}.
\end{remark}

\section{Final remarks}\label{sec:final}

\begin{remark}\label{rem:previous}
	\letpprime. Let $(M,\omega)$ be a $p$-adic analytic symplectic manifold. The study of such manifolds was initiated by Voevodsky, Warren and the second author in \cite[section 7]{PVW} about ten years ago. Recently we proved \cite[Theorem B]{CrePel-Darboux} a $p$-adic Darboux's Theorem showing that all such manifolds are locally equivalent, as it occurs in the real case. Furthermore, using this result as a stepping stone we derived a symplectic version of Serre's classification of $p$-adic analytic manifolds (provided that these manifolds are second-countable) \cite[Theorem D]{CrePel-Darboux}. In our other papers in the subject \cite{CrePel-JC,CrePel-williamson,CrePel-nonsqueezing,CrePel-angmom} we studied $p$-adic analogs of integrable systems and rigidity/flexibility questions.
\end{remark}

\begin{remark}\label{rem:counterexample-real}
	Theorem \ref{thm:weak} and Theorem \ref{thm:hamiltonian} are both false in real symplectic geometry. Indeed, a proper action with isotropic orbits may not even be weakly Hamiltonian, as shown by the following example: the rotation action of the circle on the torus is symplectic and it does not admit a momentum map, that is, it is not weakly Hamiltonian. It generalizes to higher dimensions as well: if we identify the circle with $\R/\Z$, the action of $(\R/\Z)^n$ on $(\R/\Z)^{2n}$ given by
	\[\psi((a_1,\ldots,a_n),(x_1,y_1,\ldots,x_n,y_n))=(x_1+a_1,y_1,\ldots,x_n+a_n,y_n)\]
	has Lagrangian orbits, but it is not Hamiltonian, because it has no fixed points.
\end{remark}

\begin{remark}\label{rem:proofs}
	The reason why the proof of Theorem \ref{thm:weak} does not work in the real case is because this proof only works when the balls in which the manifold is decomposed are disjoint. This is not possible with real symplectic manifolds because they are not strictly paracompact and cannot, in general, be written as a disjoint union of balls.
	
	The reason why the proof of Theorem \ref{thm:hamiltonian} does not work is more subtle. Even if the balls $B_i$, or the open sets $U_q$ of the covering, were not disjoint, the general scheme of the proof still works. The problem in the real case occurs when trying to extend the locally constant function $\tau:U_q'\to\g^*$ to $U_q$. In order to do that, we need to split $U_q$ into open sets containing the parts of $U_q'$ where $\tau$ is constant, which is possible if and only if there are no two points of $U_q'$ in the same connected component of $U_q$ with a different value of $\tau$. Hence, in the $p$-adic case it is always possible because $U_q$ is totally disconnected, while in general it is not possible in the real case.
\end{remark}

\begin{remark}\label{rem:references}
	We recommend \cite{AloHoh,Eliashberg,MarRat,McDSal,Palmer,Pelayo-hamiltonian,RTVW,Weinstein-symplectic} for further references on (real) symplectic geometry and its connections to integrable systems, topology and group actions. See also \cite{DDK,DjoDra,Dragovich-quantum,Dragovich-harmonic,RTVW,VlaVol} for uses of the $p$-adic numbers in quantum mechanics and \cite{BreFre,CLH,DKKV,DKKVZ,GKPSW} for other uses of the $p$-adic numbers in mathematical physics. For a construction of $p$-adic symplectic vector spaces, see \cite{HuHu,Zelenov}.
\end{remark}

\appendix
\section{$p$-adic numbers and $p$-adic analytic manifolds}\label{sec:appendix}

In this section we explain the concepts we need about $p$-adic geometry. We recommend \cite{Lurie,Schneider,SchWei} for references.

\letpprime. The \emph{$p$-adic absolute value} of an integer $n$ is defined as \[|n|_p=p^{-k},\] where $p^k$ is the highest power of $p$ which divides $n$. This definition is extended to $\Q$ as
\[\left|\frac{m}{n}\right|_p=\frac{|m|_p}{|n|_p}.\]
The field $\Qp$ is the metric completion of $\Q$ with respect to the $p$-adic absolute value.

Given $x\in\Qp$, we can write
\[x=\sum_{i=k}^\infty a_ip^i,\]
where $0\le a_i\le p-1$ and $k\in\Z$ is such that $p^{-k}=|x|_p$. This $k$ is denoted as $\ord_p(x)$.

A \textit{$p$-adic power series} in $(\Qp)^d$ is a series of the form
\[\sum_{I=(i_1,\ldots,i_d)\in \N^d}a_I(x_1-x_{01})^{i_1}\ldots(x_d-x_{0d})^{i_d},\]
where $(x_{01},\ldots,x_{0d})\in(\Qp)^d$ and $a_I\in\Qp$. If $U\subset (\Qp)^{d_1}$ and $V\subset(\Qp)^{d_2}$ are open sets, a function $f:U\to V$ is \textit{$p$-adic analytic} if $U$ can be expressed as
$U=\bigcup_{i\in I}U_i$
where $U_i$ is an open subset of $U$, and there are power series $f_i$ converging in $U_i$ such that $f(x)=f_i(x)$ for every $x\in U_i$.

Let $M$ and $N$ be $p$-adic analytic manifolds of dimensions $d_1$ and $d_2$ respectively. A map $F:M\to N$ is \textit{$p$-adic analytic} if, for any $m\in M$, there exist neighborhoods $U_\phi$ of $m$ and $U_\psi$ of $F(m)$ such that $\psi\circ F\circ\phi^{-1}$ is $p$-adic analytic (this composition is a function from a subset of $(\Qp)^{d_1}$ to a subset of $(\Qp)^{d_2}$).

The concepts of \textit{$p$-adic analytic function, tangent vector, $p$-adic analytic vector field} and \emph{$p$-adic analytic $k$-form} are defined by analogy with the real case. The space of analytic maps $M\to\Qp$ is denoted by $\Omega^0(M)$, and more generally, the space of $k$-forms is denoted by $\omega^k(M)$. The space of tangent vectors to $m$ is denoted by $\mathrm{T}_mM$ and the space of vector fields by $\X(M)$.

The \textit{pullback} $F^*(\alpha)\in\Omega^k(M)$ of a form $\alpha\in\Omega^k(N)$ by $F$, the \textit{push-forward} $F_*(v)\in\mathrm{T}_mN$ of a vector $v\in \mathrm{T}_mM$, and, if $F$ is bi-analytic, the \textit{push-forward} $F_*(X)\in\X(N)$ of a vector field $X\in \X(M)$, are defined analogously to the real case. Similarly for the \textit{wedge operator} and the \textit{differential operator} $\dd$.

\letnpos. \letpprime. As proposed in \cite[Section 7.2]{PVW}, a \textit{$2n$-dimensional $p$-adic analytic symplectic manifold} is a pair $(M,\omega)$, where $M$ is a $2n$-dimensional $p$-adic analytic manifold and $\omega$ is a closed non-degenerate $p$-adic analytic $2$-form. This means that $\dd\omega=0$ and that for all $m\in M$ and $u\in\mathrm{T}_mM,u\ne 0,$ there is $v\in\mathrm{T}_mM$ such that $\omega(u,v)\ne 0$. This form $\omega$ is called a \emph{$p$-adic analytic symplectic form} or simply a \emph{$p$-adic symplectic form.}

A submanifold $N$ of $(M,\omega)$ is \emph{isotropic} if, for all $m\in N$ and $u,v\in\T_m N$, we have that $\omega(u,v)=0$.

Recall the following standard notions. Let $G$ be any group and let $X$ be a topological space. An action $\psi:G\times X\to X$ on $X$ is \emph{proper} if the map $\psi:G\times X\to X$ is proper, that is, if for all convergent sequences $(m_i)_i$ of elements in $X$ and $(g_i)_i$ of elements in $G$ such that $(m_i)_i$ and $(\psi(g_i,m_i))_i$ both converge, the sequence $(g_i)_i$ contains a convergent subsequence. This always happens if $G$ is compact. The \emph{stabilizer} of $m\in X$ is the subgroup $\{g\in G:\psi(g,m)=m\}$ and its \emph{orbit} is the set $\psi(G,m)=\{\psi(g,m):g\in G\}$.

\letpprime. Let $G$ be a $p$-adic analytic Lie group and $M$ a $p$-adic analytic manifold. Two actions $\psi:G\times M\to M$ and $\psi':G'\times M'\to M'$ are \emph{weakly isomorphic} if there exists a group isomorphism $\phi:G\to G'$ and a $p$-adic analytic diffeomorphism $F:M\to M'$ such that $F(\psi(g,m))=\psi'(\phi(g),F(m))$ for every $g\in G$ and $m\in M$. If $G=G'$, the actions are \emph{isomorphic} if they are weakly isomorphic and $\phi$ is the identity.

\end{document}